\documentclass[preprint,12p]{elsarticle}
\usepackage{amsmath, amsthm, amssymb, url, mathtools}
\usepackage[linesnumbered,ruled,vlined]{algorithm2e}

\newtheorem{theorem}{Theorem}[section]
\newtheorem{lemma}[theorem]{Lemma}
\newtheorem{proposition}[theorem]{Proposition}

\theoremstyle{definition}
\newtheorem{definition}[theorem]{Definition}
\newtheorem{example}[theorem]{Example}
\theoremstyle{remark}
\newtheorem*{remark}{Remark}

\newcommand{\argmax}{\mathop{\rm arg\,max}}
\newcommand{\Ch}{\mathrm{Ch}}
\newcommand{\ot}{\coloneqq}
\newcommand{\SM}{\mathop{\mathrm{SM}}}
\newcommand{\MM}{\mathop{\mathrm{MM}}}

\newcommand{\cF}{{\mathcal F}}
\newcommand{\cM}{{\mathcal M}}

\newcommand{\hsucc}{\mathbin{\hat{\succ}}}
\newcommand{\tE}{\tilde{E}}
\newcommand{\tF}{\tilde{F}}
\newcommand{\tG}{\tilde{G}}
\newcommand{\tM}{\tilde{M}}
\newcommand{\tV}{\tilde{V}}
\newcommand{\tw}{\tilde{w}}

\journal{Discrete Applied Mathematics}

\begin{document}
\begin{frontmatter}
\title{Antimatroids Induced by Matchings}
\author{Yasushi Kawase\fnref{YK}}
\ead{kawase.y.ab@m.titech.ac.jp}
\address[YK]{Tokyo Institute of Technology, Tokyo 152-8550, Japan.}

\author{Yutaro Yamaguchi\fnref{YY}\corref{cor}}
\ead{yutaro\_yamaguchi@ist.osaka-u.ac.jp}
\address[YY]{Osaka University, Osaka 565-0871, Japan.}
\cortext[cor]{Corresponding author.}

\begin{abstract}
We explore novel connections between antimatroids and matchings in bipartite graphs.
In particular, we prove that a combinatorial structure induced by stable matchings or maximum-weight matchings is an antimatroid.
Moreover, we demonstrate that every antimatroid admits such a representation by stable matchings and maximum-weight matchings.
\end{abstract}

\end{frontmatter}

\section{Introduction}
An \emph{antimatroid} is a combinatorial abstraction of the convexity in geometry, 
which is represented by a nonempty set system $(E, \cF)$ satisfying
(i) \emph{accessibility}: every nonempty $X \in \cF$ has an element $e \in X$ such that $X-e\in\cF$,
and
(ii) \emph{union-closedness}: $X\in\cF$ and $Y\in\cF$ imply $X \cup Y \in \cF$.
See, e.g., \cite{dietrich1989,KLS1991} for more on the basics.

An antimatroid is known to be equivalent to a \emph{convex geometry} by complementation,
i.e., for any antimatroid $(E, \cF)$, the family $\{\, E \setminus X \mid X \in \cF \,\}$ forms a convex geometry (and vice versa).
A typical example of convex geometries is so-called the \emph{convex shelling} on a point set in an Euclidean space.
To establish a representation theorem for convex geometries,
Kashiwabara, Nakamura, and Okamoto \cite{KNO2005} introduced an extended procedure called the \emph{generalized convex shelling},
which can yield an arbitrary convex geometry (up to isomorphism).
Koshevoy \cite{koshevoy1999} pointed out the equivalence between \emph{path-independent choice functions}\footnotemark{} and convex geometries (formally, see Theorem~\ref{thm:Koshevoy}), which can be regarded as another representation of convex geometries.
\footnotetext{For a finite set $E$, a function $\Ch \colon 2^E \to 2^E$ is called a \emph{choice function} if $\Ch(X)\subseteq X$ holds for every $X\subseteq E$. A choice function $\Ch\colon 2^E\to 2^E$ is \emph{path-independent} if $\Ch(\Ch(X) \cup Y) = \Ch(X \cup Y)$ for every $X, Y \subseteq E$.}
More precisely, a choice function is path-independent if and only if it is the extreme function of the closure function of some convex geometry.
The \emph{substitutability}, a weaker condition than the path-independence, also yields a convex geometry \cite{fuji2015}.

In this paper, we explore a novel representation of antimatroids
with the aid of matchings in bipartite graphs.
In particular, we focus on \emph{stable matchings} and \emph{maximum-weight matchings}.

Since the seminal paper by Gale and Shapley \cite{GS1962}, the stable matching and its generalizations have been widely studied in mathematics, economics, and computer science; see, e.g, \cite{GI1989,RS1991,manlove2013} for more detail.
As a noteworthy connection to another discrete structure,
Conway \cite{knuth1976} pointed out that the set of stable matchings forms a distributive lattice under a natural dominance relation.
Conversely, Blair \cite{blair1984} proved that every finite distributive lattice is isomorphic to such one formed by the stable matchings in some instance.

The maximum-weight matching problem is one of the most fundamental combinatorial optimization problems on graphs;
see, e.g., \cite{KV2012,schrijver2003} for more detail.
It is known that maximum-weight matchings induce a \emph{valuated matroid}\footnotemark{} structure \cite{murota2009},
and every valuated matroid yields a path-independent choice function \cite{FT2006,MY2015}.
\footnotetext{
A \emph{valuated matroid} is a pair of a finite set $E$ and a function $f \colon 2^E \to \mathbb{R}\cup\{-\infty\}$ that satisfies the following \emph{gross substitute} condition~\cite{KC1982,FY2003}:
for any $p,q\in\mathbb{R}^E$ with $p\le q$ and any $X\in\argmax\{\, f(S)-\sum_{e\in S}p_e\mid S\subseteq E \,\}$,
there exists a set $Y\in\argmax\{\, f(S)-\sum_{e\in S}q_e\mid S\subseteq E \,\}$ such that $\{\, e\in X\mid p_e=q_e \,\}\subseteq Y$.}
Yokoi~\cite{yokoi2014} introduced the concept of {\em matroidal} choice functions,
which is a further restriction of path-independent choice functions,
and showed that a valuated matroid also induces a matroidal choice function under a certain condition.
However, the converse directions are not true, i.e.,
some path-independent or matroidal choice function cannot be induced
by any valuated matroid in such ways.

Our result is a representation theorem for antimatroids, which is roughly stated as follows.
A (stable or maximum-weight) matching instance
induces a map from the one-side vertex set to the other side (see Definitions \ref{def:SM} and \ref{def:MM}).
The codomain of such a map always forms an antimatroid (Theorems \ref{theorem:sm_induce} and \ref{theorem:mm_induce}), and, conversely,
any antimatroid can be obtained as the codomain of such a map induced by some matching instance
(Theorems \ref{theorem:sm_rep} and \ref{theorem:mm_rep}).
We hope that this result helps us to understand discrete structures
such as antimatroids, matchings, choice functions, and valuated matroids.

The rest of this paper is organized as follows.
In Section \ref{sec:preliminaries}, we describe necessary definitions, basic properties, and examples.
In Section \ref{sec:representations}, we show that any antimatroid admits a matching representation.
In Section \ref{sec:induce}, we prove that any matching instance induces an antimatroid.

\section{Preliminaries}\label{sec:preliminaries}

We consider matchings in a bipartite graph $G=(U,V;E)$, where $U$ and $V$ are the disjoint vertex sets (we regard $U$ and $V$ as {\em left} and {\em right}, respectively) and $E\subseteq U\times V$ is the set of edges.
For a vertex $r\in U\cup V$, we denote by $N_G(r)$ the set of neighbors of $r$, i.e., $N_G(r)=\{\, t\in U\cup V\mid (t,r)\in E~\text{or}~(r,t)\in E \,\}$.
For a vertex subset $X\subseteq U\cup V$, define $E[X] \coloneqq \{\, (u,v)\in E\mid u\in X~\text{and}~v\in X \,\}$
and $G[X]\coloneqq(U\cap X, V\cap X;E[X])$.
For an edge subset $F \subseteq E$,
define $\partial(F) \coloneqq \bigcup_{(u,v)\in F}\{u,v\}$.
An edge subset $M\subseteq E$ is called a \emph{matching} in $G$
if no two edges in $M$ have a common vertex, i.e.,
$\bigl|\bigl\{\;\!(u,v)\in M\mid u=r~\text{or}~v=r\;\!\bigr\}\bigr|\le 1$ for every vertex $r\in U \cup V$
(or, equivalently, $|\partial(M)| = 2|M|$).
For a matching $M$ in $G$ and an edge $(u,v)\in M$,
let $M(u)\coloneqq v$ and $M(v)\coloneqq u$.
We write $\mathcal{M}_{G}$ for the set of all matchings in $G$.

Let $F$ be a map from $2^U$ to $2^V$ that is induced by stable matchings or maximum-weight matchings as we will see below.
Our purpose is to study the structure of the codomain $\{\, F(U')\mid U' \subseteq U \,\}$.

\subsection{Stable Matchings}
Let us consider a bipartite graph $G=(U,V;E)$ with preferences (strict orders) $\succ_r$ on $N_G(r)$ for all $r\in U\cup V$.
We denote the profile $(\succ_r)_{r \in U \cup V}$ of preferences simply by $\succ$,
and refer to a pair $(G, \succ)$ as a \emph{stable matching instance}.
For a vertex subset $X \subseteq U \cup V$, we mean by $(G, \succ)_X$
the stable matching instance $(G[X], (\succ_r)_{r \in X})$
obtained by restricting $(G, \succ)$ to $X$.

Let $M \subseteq E$ be a matching in $G$.
An edge \((u,v)\in E\) is called a \emph{blocking pair} against $M$ in $G$ if
$[\, u \not\in \partial(M)$ or $v\succ_u M(u) \,]$ and 
$[\, v \not\in \partial(M)$ or $u\succ_v M(v) \,]$.
A matching \(M\in\mathcal{M}_{G}\) is called a \emph{stable matching} 
if there exists no blocking pair against $M$ in $G$.
It is well-known that, for any stable matching instance,
there exists at least one stable matching,
and moreover all stable matchings consist of the same set of vertices.

\begin{theorem}[McVitie--Wilson \cite{MW70}]\label{thm:RH}
  For any stable matching instance,
  if two matchings $M_1$ and $M_2$ in it are both stable,
  then $\partial(M_1) = \partial(M_2)$.
\end{theorem}

By this property, for each subset $U' \subseteq U$,
the set of vertices in $V$ who are matched in a stable matching in $(G, \succ)_{U' \cup V}$ is uniquely determined.
This fact naturally defines a map from $2^U$ to $2^V$.

For sake of disambiguation,
we adopt specific stable matchings defined algorithmically as follows.
For a stable matching instance,
a stable matching can be obtained by a simple algorithm,
so-called the {\em deferred acceptance algorithm} \cite{GS1962,MW1971} (see Algorithm~\ref{alg:gs}).
In each iteration, an unmatched left vertex $u$ proposes to the most-preferred right vertex $v$ in $u$'s preference list to whom it hasn't yet proposed.
Then, $v$ accepts the proposal if $v$ is unmatched or prefers $u$ to the current partner $u'$ (in this case, $u'$ becomes unmatched).
Otherwise, i.e., if $v$ prefers the current partner $u'$ to $u$, the proposal is rejected.
The process is repeated until every left vertex is matched or rejected by all its neighbors.

\begin{algorithm}
  \caption{Deferred Acceptance Algorithm}\label{alg:gs}
  \SetKwInOut{Input}{Input}\Input{A bipartite graph $G=(U,V;E)$ and a preference profile ${\succ} = (\succ_r)_{r\in U\cup V}$}
  \SetKwInOut{Output}{Output}\Output{A stable matching $\SM(G, \succ)\subseteq E$}
  let $T\ot U$ 
  and $M\ot\emptyset$\;
  \lForEach{$u\in U$}{$R_u\ot N_{G}(u)$}
  \While{$T\ne\emptyset$}{
    pick $u\in T$ arbitrarily\;
    \lIf{$R_u=\emptyset$}{$T\ot T-u$}
    \Else{
      take $v\in R_u$ so that $v\succeq_u v'$ for all $v'\in R_u$\;
      \lIf{$v \not\in \partial(M)$}{$M\ot M+(u,v)$, $T\ot T-u$}
      \Else{
        let $u' \ot M(v)$\;
        \lIf{$u'\succ_v u$}{$R_u\ot R_u-v$}
        \lElse{
          $M\ot M+(u,v)-(u',v)$,
          $R_{u'}\ot R_{u'}-v$,
          and $T\ot T+u'-u$}
      }
    }
  }
  \Return $M$\;
\end{algorithm}

A significant feature of this algorithm is that the output does not depend on the order of proposals,
i.e., it defines a unique stable matching.
For a stable matching instance $(G = (U, V; E), \succ)$ and a subset $U' \subseteq U$,
we denote by $\SM(G, \succ; U')$ the output of the deferred acceptance algorithm
for the restricted instance $(G, {\succ})_{U' \cup V}$.

\begin{definition}\label{def:SM}
  The map $F\colon 2^U\to 2^V$ {\em induced by a stable matching instance} $(G=(U,V;E), \succ)$ is defined by
    \[F(U') \coloneqq \partial(\SM(G, \succ; U')) \cap V = \{\, v\mid (u,v)\in\SM(G, \succ; U') \,\} \quad (U' \subseteq U).\]
\end{definition}

Note again that, by Theorem~\ref{thm:RH}, 
one can replace $\SM(G, \succ; U')$ in the above definition
with an arbitrary stable matching in the restricted instance $(G, {\succ})_{U' \cup V}$.

Let us mention two important properties of $F$
(which can be observed by picking $u$ from $U_2$ as priority in Line 4 in Algorithm~\ref{alg:gs}).
\begin{lemma}\label{lemma:sm_prop}
  The map $F\colon 2^U\to 2^V$ induced by a stable matching instance satisfies the following. 
  \renewcommand{\theenumi}{$(\alph{enumi})$}
  \renewcommand{\labelenumi}{\theenumi}
  \begin{enumerate}
  \item If $U_2\subseteq U_1\subseteq U$, then $F(U_2)\subseteq F(U_1)$. \label{lemma:sm_prop_monotone}
  \item If $U_2\subseteq U_1\subseteq U$ and $|F(U_1)|=|U_1|$, then $|F(U_2)|=|U_2|$. \label{lemma:sm_prop_shrink} 
  \end{enumerate}
\end{lemma}

We give an example of the map induced by a stable matching instance.
\begin{example}
Suppose that $U=\{u_1,u_2,u_3\}$, $V=\{v_1,v_2,v_3\}$, and
  \[E=\{(u_1,v_1),(u_1,v_2),(u_2,v_1),(u_2,v_3),(u_3,v_1),(u_3,v_2)\}.\]
Consider an instance $(G=(U,V;E), \succ)$, where
\begin{align*}
 &u_1\colon~ v_1\succ_{u_1} v_2, &&v_1\colon~ u_3\succ_{v_1} u_2\succ_{v_1} u_1,\\
 &u_2\colon~ v_1\succ_{u_2} v_3, &&v_2\colon~ u_1\succ_{v_2} u_3,\\
 &u_3\colon~ v_2\succ_{u_3} v_1, &&v_3\colon~ u_2.
\end{align*}
Then, for example, $F(U)=\{v_1,v_2,v_3\}$ because $\SM(G, \succ; U)=\{(u_1,v_2),\allowbreak{}(u_2,v_3),\allowbreak{}(u_3,v_1)\}$.
By similar calculations, we obtain that the codomain of $F$ is
\[\{\, F(U')\mid U'\subseteq U \,\}=\{\emptyset,\{v_1\},\{v_2\},\{v_1,v_2\},\{v_1,v_2,v_3\}\},\]
which forms an antimatroid on $V$.
\end{example}

\subsection{Maximum-Weight Matchings}
Given a bipartite graph $G=(U,V;E)$ with weights $w\colon E\to\mathbb{R}$,
the weight of a matching $M$, denoted by $w(M)$, is defined to be the sum of the weights of the edges in $M$.
We refer to a pair $(G,w)$ as a \emph{weighted matching instance}.
A \emph{maximum-weight matching} in $G$ is a matching in $\cM_G$ with weight $\max_{M \in \mathcal{M}_{G}} w(M)$.
If there exist multiple maximum-weight matchings,
we pick the lexicographically smallest (with respect to a fixed order on the edges) one among them.
Throughout the paper, we assume that a maximum-weight matching is always determined uniquely in this sense.
For each $U' \subseteq U$, let $\MM(G,w; U')$ denote the unique maximum-weight matching in $G[U' \cup V]$. 
\begin{definition}\label{def:MM}
  The map $F\colon 2^U\to 2^V$ {\em induced by a weighted matching instance} $(G=(U,V;E),w)$ is defined by
    \[F(U') \coloneqq \partial(\MM(G, w; U')) \cap V = \{\, v\mid (u,v)\in \MM(G, w; U') \,\} \quad (U' \subseteq U).\]
\end{definition}

The map $F$ induced by a weighted matching instance has the same properties as in Lemma \ref{lemma:sm_prop}.
\begin{lemma}\label{lemma:mm_prop}
  The map $F\colon 2^U \to 2^V$ induced by a weighted matching instance $(G = (U, V; E), w)$ satisfies the following. 
  \renewcommand{\theenumi}{$(\alph{enumi})$}
  \renewcommand{\labelenumi}{\theenumi}
  \begin{enumerate}
  \item If $U_2\subseteq U_1\subseteq U$, then $F(U_2)\subseteq F(U_1)$. \label{lemma:mm_prop_monotone}
  \item If $U_2\subseteq U_1\subseteq U$ and $|F(U_1)|=|U_1|$, then $|F(U_2)|=|U_2|$. \label{lemma:mm_prop_shrink}
  \end{enumerate}
\end{lemma}

To prove the lemma, we use the following property
(which follows from the fact that at most two edges in $M_1 \triangle M_2$ are incident to each vertex).
\begin{proposition}[cf.~{\cite[Chapter 16]{schrijver2003}}]\label{prop:sd}
  For $M_1,M_2\in\mathcal{M}_G$, the symmetric difference $M_1 \triangle M_2 = (M_1 \setminus M_2) \cup (M_2 \setminus M_1)$ forms disjoint cycles and paths.
\end{proposition}

Now, we turn to the proof of Lemma~\ref{lemma:mm_prop}.
\begin{proof}[Proof of Lemma~\ref{lemma:mm_prop}]
\noindent$\boldsymbol{(a)}$: To prove by contradiction, suppose that $F(U_2)\not\subseteq F(U_1)$ and let $v^*\in F(U_2)\setminus F(U_1)$.
Define $M_i \coloneqq \MM(G,w; U_i)$ for $i = 1, 2$.
By Proposition~\ref{prop:sd}, $M_1 \triangle M_2$ forms disjoint cycles and paths.
Since $v^\ast \in F(U_2) \setminus F(U_1)$,
the connected component of $G[M_1 \triangle M_2]$ containing $v^\ast$ is a path of length at least 1,
and let $P\subseteq M_1\triangle M_2$ be the set of edges in the path.
Then, $M_i \triangle P \in \cM_{G[U_i \cup V]}$ $(i = 1, 2)$
and $w(M_1)+w(M_2)=w(M_1\triangle P)+w(M_2\triangle P)$.
Hence we have $w(M_i)=w(M_i\triangle P)$
and $M_i$ is lexicographically smaller than $M_i \triangle P$, for $i=1,2$, as $M_i$ is the lexicographically smallest maximum weight matching in $\mathcal{M}_{G[U_i\cup V]}$.
However, the latter property implies that $M_i\cap P$ is lexicographically smaller than $M_j\cap P$ for both $(i,j)=(1,2),(2,1)$, a contradiction.

\medskip
\noindent$\boldsymbol{(b)}$: It is sufficient to prove the case when $U_1 = U_2 + q$ for some $q \in U \setminus U_2$. 
Define $M_i \coloneqq \MM(G,w; U_i)$ for $i = 1, 2$.
By Proposition~\ref{prop:sd}, $M_1 \triangle M_2$ forms disjoint cycles and paths,
and moreover it consists of a single path from $q$ (or the empty set, which can be regarded as a path of length $0$),
since otherwise (i.e., if it contains a cycle or a path disjoint from $q$ that is of length at least $1$)
we can improve at least one of $M_1$ and $M_2$.
Therefore, $|U_2| + 1 = |U_1| = |F(U_1)| = |M_1| \leq |M_2| + 1 = |F(U_2)| + 1 \leq |U_2| + 1$, in which the equalities must hold throughout. 
\end{proof}

Let us see an example of the map induced by a weighted matching instance.
\begin{example}
Suppose that $U=\{u_1,u_2,u_3\}$, $V=\{v_1,v_2\}$, and
  \[E=\{(u_1,v_1),(u_1,v_2),(u_2,v_1),(u_3,v_1)\}.\]
Consider an instance $(G=(U,V;E),w)$, where
\[  w((u_1,v_1))=20,~
  w((u_1,v_2))=8,~
  w((u_2,v_1))=9,~\text{and}~
  w((u_3,v_1))=15.\]
Then, for example, $F(\{u_1,u_2\})=\{v_1\}$ because $\MM(G,w; \{u_1,u_2\})=\{(u_1,v_1)\}$.
By similar calculations, we obtain that the codomain of $F$ is
\[\{\, F(U')\mid U'\subseteq U \,\}=\{\emptyset,\{v_1\},\{v_1,v_2\}\},\]
which forms an antimatroid on $V$.
\end{example}

\begin{remark}
While in this paper we focus on the set of matched right vertices when restricting the left vertices, readers may get interested also in the set of matched left vertices.
More precisely, for a stable matching instance $(G, {\succ})$ and a maximum-weight matching instance $(G,w)$, 
one may consider choice functions
$\Ch \colon 2^U \to 2^U$ defined by
\(\Ch(U') \coloneqq \partial(\SM(G, {\succ}; U')) \cap U\) and 
\(\Ch(U') \coloneqq \partial(\MM(G, w; U')) \cap U\) \((U' \subseteq U)\), respectively.
Kawase~\cite{kawase2015} proved that, in both cases, the function $\Ch$ is not only path-independent but also \emph{size-monotone}, i.e., $|\Ch(U_1)|\leq|\Ch(U_2)|$ for all $U_1\subseteq U_2\subseteq U$.
The size-monotonicity implies that the antimatroids yielded by such choice functions are restricted, and we cannot obtain a representation theorem for antimatroids in this way.
\end{remark}

\section{Matching Representations of Antimatroids}\label{sec:representations}
In this section, we provide a representation of an antimatroid by a matching instance.
It is worth remarking that the number of distinct antimatroids on a fixed $n$-element set
is doubly exponential in $n$ (see Appendix~\ref{sec:number}).
Hence, we need an exponential-size representation (i.e., a representation with exponentially many bits).

Let $(S,\mathcal{F})$ be an antimatroid.
Let $d\colon\mathcal{F}\setminus\{\emptyset\}\to S$ be a function such that $d(X)\in X$ and $X-d(X)\in\mathcal{F}$ for every $X\in\mathcal{F}\setminus\{\emptyset\}$.
There exists such a function $d$ since $\mathcal{F}$ satisfies accessibility.

Let $\succ^*$ be a total order on $\mathcal{F}\setminus\{\emptyset\}$ such that $X\succ^* Y$ whenever $X\subsetneq Y$.
Namely, $X\succ^* Y$ implies $Y\not\subseteq X$.
Also, let $\succ^{X}$ be the order on each $X=\{a_1,\dots,a_{k}\}\in\mathcal{F}\setminus\{\emptyset\}$ such that $a_1\succ^{X}\dots\succ^{X}a_{k}$, where $a_i=d(X \setminus \{a_{i+1},\dots,a_{k}\})$ $(i=1,2,\dots,k)$.
Note that $\{a_1,\dots,a_i\}\in\mathcal{F}$ $(i = 0, 1, \dots, k)$ by the definition of $d$.

\subsection{Representation by Stable Matchings}
We construct a stable matching instance $(G=(U,V;E),\succ)$ as follows:
\begin{itemize}
\item $U\coloneqq \mathcal{F}\setminus\{\emptyset\}$ and $V\coloneqq S$;
\item $E\coloneqq \{\, (u,v)\mid u\in U,~v\in u \,\}$;
\item ${\succ_u}\coloneqq{\succ^{u}}$ for each $u\in U$;
\item let $\succ_v$ be the restriction of $\succ^*$ to $\{\, u \mid v\in u \,\}$ for each $v\in V$.
\end{itemize}
We prove that this stable matching instance derives the desired antimatroid.
\begin{theorem}\label{theorem:sm_rep}
Let $(S, \cF)$ be an antimatroid, and $F \colon 2^{\cF \setminus \{\emptyset\}} \to 2^S$
the map induced by the stable matching instance $(G, \succ)$ defined as above.
Then the codomain of $F$ coincides with $\cF$.
\end{theorem}
\begin{proof}
Let $\mathcal{F}'=\{\, F(U')\mid U'\subseteq U \,\}$. 
We claim that $\mathcal{F}'=\mathcal{F}$.

We first see $\mathcal{F}\subseteq \mathcal{F}'$.
Let $X=\{v_1,\dots,v_{k}\}\in\mathcal{F}$ such that $v_1\succ_{X}\dots\succ_{X}v_{k}$, i.e., $v_i=d(X \setminus \{v_{i+1},\dots,v_{k}\})$ $(i=1,2,\dots,k)$.
We define $u_i \coloneqq  \{v_1,\dots,v_i\} \in \cF \setminus \{\emptyset\} = U$ for $i=1,2,\dots,k$.
Then, $X=F(\{u_1,\dots,u_k\})\in\mathcal{F}'$ because $\SM(G,\succ; \{u_1,\dots,u_k\})=\{(u_1,v_1),\dots,(u_k,v_k)\}$ by $u_1\succ^*\dots\succ^* u_k = X$ and $v_1\succ_{u_i}\dots\succ_{u_i}v_{i}$ ($i=1,2,\dots,k$).

Next, we prove the converse direction, i.e., $\mathcal{F}'\subseteq \mathcal{F}$.
Let $U'\subseteq U$ and $M\coloneqq \SM(G, \succ; U')$. 
Since each $u \in U'$ matched with someone in $M$ proposes only to the neighbors $v' \in N_{G[U' \cup V]}(u) = u$ such that $v' \succeq_u M(u)$
throughout the deferred acceptance algorithm (recall Algorithm \ref{alg:gs}),
we have
$$F(U')=\partial(M) \cap V=\bigcup_{(u,v)\in M}\{\, v' \mid v' \in u,~v'\succeq_{u} v \,\}\in \mathcal{F},$$
where the last membership follows from the facts that
$\{\, v' \mid v' \in u,~v'\succeq_{u} v \,\}\in\mathcal{F}$ for all $(u,v)\in E$
(recall the definitions of ${\succ_u} = {\succ^u}$ and of $d$)
and that $\mathcal{F}$ is union-closed.
Therefore, we get $\mathcal{F}'\subseteq \mathcal{F}$.
\end{proof}

\subsection{Representation by Maximum-Weight Matchings}
We define $b$ to be a unique bijection from $\mathcal{F}\setminus\{\emptyset\}$ to $\{1,2,\dots,|\mathcal{F}| - 1\}$
consistent with $\succ^\ast$, i.e.,
for any distinct $X,Y\in \mathcal{F}\setminus\{\emptyset\}$,
we have $b(X)>b(Y)$ if and only if $X\succ^* Y$.
We build a weighted matching instance $(G=(U,V;E),w)$ as follows:
\begin{itemize}
\item $U\coloneqq \mathcal{F}\setminus\{\emptyset\}$ and $V\coloneqq S$;
\item $E\coloneqq \{\, (u,v)\mid u\in U,~v\in u \,\}$;
\item $w((u,v_i))\coloneqq2^{|V|\cdot b(u)+i}$ for $u=\{v_1,\dots,v_k\}\in\mathcal{F}\setminus\{\emptyset\}$ and $v_i\in u$
such that $v_1\succ^u\dots\succ^u v_k$.
\end{itemize}
We show that this weighted matching instance also derives the desired antimatroid.
Note that the maximum-weight matching is lexicographically maximum (with respect to the order of edge weights)
because the edge weights are distinct power-of-two values.
\begin{theorem}\label{theorem:mm_rep}
Let $(S, \cF)$ be an antimatroid, and $F \colon 2^{\cF \setminus \{\emptyset\}} \to 2^S$
the map induced by the weighted matching instance $(G, w)$ defined as above.
Then the codomain of $F$ coincides with $\cF$.
\end{theorem}
\begin{proof}
Let $\mathcal{F}'=\{\, F(U')\mid U'\subseteq U \,\}$. 
We claim that $\mathcal{F}'=\mathcal{F}$.

We first see $\mathcal{F}\subseteq \mathcal{F}'$.
Let $X=\{v_1,\dots,v_{k}\}\in\mathcal{F}$ such that $v_1\succ^{X}\dots\succ^{X}v_{k}$, i.e., $v_i=d(X \setminus \{v_{i+1},\dots,v_{k}\})$ $(i=1,2,\dots,k)$.
We define $u_i \coloneqq  \{v_1,\dots,v_i\} \in \cF \setminus \{\emptyset\} = U$ for $i=1,2,\dots,k$.
Then, $X=F(\{u_1,\dots,u_k\})\in\mathcal{F}'$ because
$\{(u_1,v_1),\dots,(u_k,v_k)\}$ is the maximum-weight matching between $\{u_1,\dots,u_k\}$ and $V$
by $b(u_1)>\dots>b(u_k)$.

Next, we prove the converse direction, i.e., $\mathcal{F}'\subseteq \mathcal{F}$.
Let $U'\subseteq U$ and $M\coloneqq \MM(G,w; U')$. 
Recall that the maximum-weight matching is lexicographically maximum.
We then have
$$F(U')=\partial(M) \cap V=\bigcup_{(u,v)\in M}\{\, v' \mid v' \in u,~v'\succeq_{u} v \,\}\in \mathcal{F},$$
since $\{\, v' \mid v' \in u,~v'\succeq^u v \,\}\in\mathcal{F}$ for all $(u,v)\in E$ and $\mathcal{F}$ is union-closed.
Therefore, we get $\mathcal{F}'\subseteq \mathcal{F}$.
\end{proof}

\section{Antimatroids Induced by Matchings}\label{sec:induce}
In this section, we prove that any 
matching instance induces an antimatroid.

\subsection{Antimatroids Induced by Stable Matchings}
\begin{theorem}\label{theorem:sm_induce}
  Let $F\colon 2^{U} \to 2^{V}$ be the map induced by a stable matching instance $(G=(U,V;E), {\succ})$,
  and $\cF \coloneqq \{\, F(U') \mid U' \subseteq U \,\}$.
  Then the set system $(V, \cF)$ is an antimatroid.
\end{theorem}
\begin{proof}
We have $\emptyset\in\cF$ since $F(\emptyset)=\emptyset$.
To see the accessibility, let us fix $V'\in\cF \setminus \{\emptyset\}$ and let $U'$ be a subset of $U$ such that $F(U')=V'$.
Define $M\coloneqq\SM(G, \succ; U')$ and $U''= \partial(M) \cap U$.
Then $M$ is also a stable matching in $G[U''\cup V]$ because $U''\subseteq U'$ and $M \in \cM_{G[U''\cup V]}$.
Thus we have $F(U'')=V'$ and $|U''|=|V'| \geq 1$.
Let us fix $u^*\in U''$. Then $F(U''-u^*)\subseteq V'$ and $F(U''-u^*)=|V'|-1$
by Lemma \ref{lemma:sm_prop}.
Therefore, there exists $v\in V'$ such that $V'-v=F(U''-u^*)\in\cF$.

In what follows, we show that $\cF$ is union-closed.
Fix any two subsets $U_1, U_2 \subseteq U$, and let $V^* \coloneqq F(U_1) \cup F(U_2)$.
We shall show that there exists $U^* \subseteq U_1 \cup U_2$ such that $V^* = F(U^*)$.
Note that, due to \ref{lemma:sm_prop_monotone} in Lemma \ref{lemma:sm_prop},
we have $F(U_1) \subseteq F(U_1 \cup U_2)$ and $F(U_2) \subseteq F(U_1 \cup U_2)$,
and hence $V^* \subseteq F(U_1 \cup U_2)$.

Let $M_i\coloneqq\SM(G,\succ; U_i)$ for $i = 1, 2$,
$\hat{E}\coloneqq\{\, (u,v)\mid (u,v')\in M_1 \cup M_2,~v\succeq_u v' \,\}$,
and $\hat{G}\coloneqq(U,V;\hat{E})$.
Note that $\partial(\hat{E}) \cap V \subseteq V^*$, because if there exists $(u,v)\in \hat{E}$ such that $v\not\in V^*$, then it is a blocking pair against $M_1$ or $M_2$.
For each $r\in U\cup V$, let $\hsucc_r$ denote the restriction of $\succ_r$ to $N_{\hat{G}}(r)$, i.e.,
$\hsucc_r$ is a strict order on $N_{\hat{G}}(r)$ such that, for every $x, y\in N_{\hat{G}}(r)$,
$x \hsucc_r y$ if and only if $x \succ_r y$.
We then have
  \[\SM(\hat{G}, \hsucc; U_i)=M_i=\SM(G, \succ; U_i) \quad (i = 1, 2).\]

Let $\hat{F}\colon 2^U\to 2^V$ be the map induced by $(\hat{G}, \hsucc)$.
Then, by \ref{lemma:sm_prop_monotone} in Lemma \ref{lemma:sm_prop},
$$F(U_i)= \partial(M_i) \cap V =\hat{F}(U_i)\subseteq \hat{F}(U_1\cup U_2) \quad (i = 1, 2),$$ 
and hence $V^* = F(U_1) \cup F(U_2)\subseteq \hat{F}(U_1\cup U_2)$.
In addition, since $\partial(\hat{E}) \cap V \subseteq V^*$, we have $\hat{F}(U_1\cup U_2)\subseteq V^*$.
Thus, we obtain $\hat{F}(U_1\cup U_2)=V^*$.

Let $\hat{M}\coloneqq\SM(\hat{G}, \hsucc; U_1\cup U_2)$ and $U^\ast\coloneqq \partial(\hat{M}) \cap U \subseteq U_1 \cup U_2$.
We define
$$M^\ast \coloneqq \SM(\hat{G}, \hsucc; U^\ast).$$
Note that $\hat{F}(U_1\cup U_2)=\hat{F}(U^\ast)$ because $\hat{M}$ is a stable matching in $\hat{G}[U^\ast\cup V]$.
In addition, every vertex in $U^\ast$ is matched in $M^\ast$
because $|M^\ast|=|\hat{F}(U^\ast)|=|\hat{F}(U_1\cup U_2)|=|\hat{M}|=|U^\ast|$.

The proof is completed by showing that $M^\ast$ is a stable matching also in $G[U^\ast\cup V]$ (with respect to $\succ$)
because this implies $F(U^\ast)=\partial(M^\ast) \cap V=\hat{F}(U^\ast)=\hat{F}(U_1\cup U_2)=V^\ast$ (with the aid of Theorem~\ref{thm:RH}).
To obtain a contradiction, suppose that there exists a blocking pair $(u^*,v^*) \in E[U^\ast\cup V]$ against $M^\ast$.
Since $M^\ast$ is a stable matching in $\hat{G}[U^\ast\cup V]$, 
we can assume that $(u^*,v^*)\not\in\hat{E}$.
Then, by the definition of $\hat{E}$, we have $v\succ_{u^*} v^*$ for every $v \in N_{\hat{G}}(u^\ast)$.
As $(u^*, M^\ast(u^*))\in\hat{E}$ implies $M^\ast(u^\ast) \succ_{u^\ast} v^\ast$,
we get that $(u^*,v^*)$ cannot be a blocking pair against $M^\ast$ in $G[U^\ast\cup V]$.
This contradicts our assumption.
\end{proof}

\subsection{Antimatroids Induced by Maximum-Weight Matchings}
\begin{theorem}\label{theorem:mm_induce}
  Let $F\colon 2^{U} \to 2^{V}$ be the map induced by a weighted matching instance $(G=(U,V;E), w)$,
  and $\cF \coloneqq \{\, F(U') \mid U' \subseteq U \,\}$. Then the set system $(V, \cF)$ is an antimatroid.
\end{theorem}
\begin{proof}
We have $\emptyset\in\cF$ since $F(\emptyset)=\emptyset$.
Also, we can derive the accessibility
from Lemma \ref{lemma:mm_prop}
similarly to the proof of Theorem \ref{theorem:sm_induce}.

It remains to prove that $\cF$ is union-closed.
Fix any two subsets $U_1, U_2 \subseteq U$, and let $V^* \coloneqq F(U_1) \cup F(U_2)$.
Let $M_i\coloneqq\MM(G, w; U_i)$ for $i = 1, 2$.
For each vertex $v\in V\setminus V^*$, we create $|U|$ new vertices $h_{v,j}$ $(j = 1, 2, \dots, |U|)$,
and let $H_v$ denote the set of those vertices.
We define a new weighted matching instance $(\tG=(U,\tV;\tE), \tw)$ as follows: 
\begin{align*}
\tV&\coloneqq V^*\cup\bigcup_{v\in V\setminus V^*}H_v,\\
\tE&\coloneqq\{\, (u,v)\mid (u,v)\in E,~v\in V^* \,\}\cup \bigcup_{j=1}^{|U|}\{\, (u,h_{v,j})\mid (u,v)\in E,~v \notin V^* \,\},
\end{align*}

and for each $(u, v) \in \tE$,
  $$\tw((u,v))\coloneqq\begin{cases}
  w((u,v))&\text{if }v\in V^*,\\
  w((u,v'))&\text{if }v\in H_{v'}.
  \end{cases}$$
Let $\tM$ be the maximum-weight matching\footnote{Recall that there may be multiple choices of maximum-weight matchings $\tM$ in $\tG[U \cup \tV]$
with respect to $\tw$,
and if so $\MM(\tG, \tw; U)$ is defined as the lexicographically smallest one among them,
with respect to a fixed order on the edges $\tE$.
One can arbitrarily fix an order on $\tE$, which is not essential in the proof.} $\MM(\tG,\tw; U)$,
$M^*\coloneqq \tM\cap E$, and $U^*\coloneqq \partial(M^*) \cap U$.
In addition, let $\tF\colon 2^U\to 2^{\tV}$ be the map induced by $(\tG,\tw)$.

We first claim that $\tF(U^*)=V^*$.
Note that $\MM(G,w; U_i)=\MM(\tG,\tw; U_i) =M_i$ for $i = 1, 2$ 
by the definition of $(\tG,\tw)$.
As $F(U_i)= \partial(M_i) \cap V=\tF(U_i)\subseteq \tF(U_1\cup U_2)$ 
by \ref{lemma:mm_prop_monotone} in Lemma \ref{lemma:mm_prop},
we have $V^* = F(U_1)\cup F(U_2)\subseteq \tF(U_1\cup U_2)$.
Thus we get $\tF(U^*)=\tF(U_1\cup U_2)\cap V^*=V^*$.

Next, we observe that $M^*=\MM(G,w; U^\ast)$.
Suppose to the contrary that $M^*\triangle\MM(G,w; U^\ast) \neq \emptyset$,
and let $X\subseteq M^*\triangle\MM(G,w; U^\ast)$ be one of its connected components.
We then have $\tw(\tM\triangle X) + w(\MM(G,w; U^\ast)\triangle X) = \tw(\tM) + w(\MM(G,w; U^\ast))$.
Then, by the definition of $\MM$,
we have
$\tw(\tM\triangle X) = \tw(\tM)$, $w(\MM(G,w; U^\ast)\triangle X) = w(\MM(G,w; U^\ast))$, and
$\tM$ and $\MM(G,w; U^\ast)$ are lexicographically smaller than
$\tM \triangle X$ and $\MM(G,w; U^\ast)\triangle X$, respectively.
The last property implies that $\tM \cap X = M^\ast \cap X$ and $\MM(G, w; U^\ast) \cap X$ are lexicographically smaller than each other,
a contradiction.

Consequently, we obtain $V^*=\tF(U^*)= \partial(M^*) \cap V=F(U^*)\in\cF$.
\end{proof}

\section{Concluding Remark}
In this paper, we have seen novel connections between matchings and antimatroids.
We have proved that the codomains of maps induced by stable matchings and maximum-weight matchings form antimatroids.
Conversely, we have provided representation theorems for antimatroids through stable matchings and maximum-weight matchings.

Although stable matchings and maximum-weight matchings have similar flavor,
the proofs of Theorems~\ref{theorem:sm_induce} and~\ref{theorem:mm_induce} are based on substantially different properties.
It is an open problem to establish a unified framework to discuss a sufficient condition
for a map $F\colon 2^U\to 2^V$ to induce an antimatroid on $V$ as its codomain.

\section*{Acknowledgments}
The first author was supported by JSPS KAKENHI Grant Number 16K16005 and
the second author was supported by JSPS KAKENHI Grant Number 16H06931.

\bibliographystyle{elsarticle-num-names-alpha}
\bibliography{am}

\begin{thebibliography}{25}
\providecommand{\natexlab}[1]{#1}
\providecommand{\url}[1]{\texttt{#1}}
\providecommand{\urlprefix}{URL }
\expandafter\ifx\csname urlstyle\endcsname\relax
  \providecommand{\doi}[1]{doi:\discretionary{}{}{}#1}\else
  \providecommand{\doi}[1]{doi:\discretionary{}{}{}\begingroup
  \urlstyle{rm}\url{#1}\endgroup}\fi
\providecommand{\bibinfo}[2]{#2}

\bibitem[{Blair(1984)}]{blair1984}
\bibinfo{author}{C.~Blair}, \bibinfo{title}{Every finite distributive lattice
  is a set of stable matchings}, \bibinfo{journal}{Journal of Combinatorial
  Theory (Series A)} \bibinfo{volume}{37}~(\bibinfo{number}{3})
  (\bibinfo{year}{1984}) \bibinfo{pages}{353--356}.

\bibitem[{Dietrich(1989)}]{dietrich1989}
\bibinfo{author}{B.~L. Dietrich}, \bibinfo{title}{Matroids and antimatroids---a
  survey}, \bibinfo{journal}{Discrete Mathematics}
  \bibinfo{volume}{78}~(\bibinfo{number}{3}) (\bibinfo{year}{1989})
  \bibinfo{pages}{223--237}.

\bibitem[{Fleiner(2003)}]{fleiner2003fixed}
\bibinfo{author}{T.~Fleiner}, \bibinfo{title}{A fixed-point approach to stable
  matchings and some applications}, \bibinfo{journal}{Mathematics of Operations
  Research} \bibinfo{volume}{28}~(\bibinfo{number}{1}) (\bibinfo{year}{2003})
  \bibinfo{pages}{103--126}.

\bibitem[{Fuji(2015)}]{fuji2015}
\bibinfo{author}{J.~Fuji}, \bibinfo{title}{Substitutable choice functions and
  convex geometry}, \bibinfo{journal}{Discrete Applied Mathematics}
  \bibinfo{volume}{186} (\bibinfo{year}{2015}) \bibinfo{pages}{283--285}.

\bibitem[{Fujishige and Tamura(2006)}]{FT2006}
\bibinfo{author}{S.~Fujishige}, \bibinfo{author}{A.~Tamura}, \bibinfo{title}{A
  general two-sided matching market with discrete concave utility functions},
  \bibinfo{journal}{Discrete Applied Mathematics} \bibinfo{volume}{154}
  (\bibinfo{year}{2006}) \bibinfo{pages}{950--970}.

\bibitem[{Fujishige and Yang(2003)}]{FY2003}
\bibinfo{author}{S.~Fujishige}, \bibinfo{author}{Z.~Yang}, \bibinfo{title}{A
  Note on Kelso and Crawford's Gross Substitutes Condition},
  \bibinfo{journal}{Mathematics of Operations Research}
  \bibinfo{volume}{28}~(\bibinfo{number}{3}) (\bibinfo{year}{2003})
  \bibinfo{pages}{463--469}.

\bibitem[{Gale and Shapley(1962)}]{GS1962}
\bibinfo{author}{D.~Gale}, \bibinfo{author}{L.~S. Shapley},
  \bibinfo{title}{College Admissions and the Stability of Marriage},
  \bibinfo{journal}{American Mathematical Monthly} \bibinfo{volume}{69}
  (\bibinfo{year}{1962}) \bibinfo{pages}{9--14}.

\bibitem[{Gusfield and Irving(1989)}]{GI1989}
\bibinfo{author}{D.~Gusfield}, \bibinfo{author}{R.~W. Irving},
  \bibinfo{title}{The Stable Marriage Problem: Structure and Algorithms},
  \bibinfo{publisher}{MIT Press}, \bibinfo{address}{Boston},
  \bibinfo{year}{1989}.

\bibitem[{Kashiwabara et~al.(2005)Kashiwabara, Nakamura, and Okamoto}]{KNO2005}
\bibinfo{author}{K.~Kashiwabara}, \bibinfo{author}{M.~Nakamura},
  \bibinfo{author}{Y.~Okamoto}, \bibinfo{title}{The affine representation
  theorem for abstract convex geometries}, \bibinfo{journal}{Computational
  Geometry} \bibinfo{volume}{30}~(\bibinfo{number}{2}) (\bibinfo{year}{2005})
  \bibinfo{pages}{129--144}.

\bibitem[{Kawase(2015)}]{kawase2015}
\bibinfo{author}{Y.~Kawase}, \bibinfo{title}{The Secretary Problem with a
  Choice Function}, in: \bibinfo{booktitle}{Proceedings of ISAAC},
  \bibinfo{pages}{129--139}, \bibinfo{year}{2015}.

\bibitem[{Kelso~Jr and Crawford(1982)}]{KC1982}
\bibinfo{author}{A.~S. Kelso~Jr}, \bibinfo{author}{V.~P. Crawford},
  \bibinfo{title}{Job matching, coalition formation, and gross substitutes},
  \bibinfo{journal}{Econometrica}  (\bibinfo{year}{1982})
  \bibinfo{pages}{1483--1504}.

\bibitem[{Knuth(1974)}]{Knuth1974}
\bibinfo{author}{D.~E. Knuth}, \bibinfo{title}{The asymptotic number of
  geometries}, \bibinfo{journal}{Journal of Combinatorial Theory, Series A}
  \bibinfo{volume}{16}~(\bibinfo{number}{3}) (\bibinfo{year}{1974})
  \bibinfo{pages}{398--400}, ISSN \bibinfo{issn}{0097--3165}.

\bibitem[{Knuth(1976)}]{knuth1976}
\bibinfo{author}{D.~E. Knuth}, \bibinfo{title}{Marriage Stables},
  \bibinfo{publisher}{Montr\'{e}al: Les Presses de l'Universit\'{e} de
  Montr\'{e}al}, \bibinfo{year}{1976}.

\bibitem[{Korte et~al.(1991)Korte, Lov{\'a}sz, and Schrader}]{KLS1991}
\bibinfo{author}{B.~Korte}, \bibinfo{author}{L.~Lov{\'a}sz},
  \bibinfo{author}{R.~Schrader}, \bibinfo{title}{Greedoids}, Algorithms and
  Combinatorics, \bibinfo{publisher}{Springer}, \bibinfo{year}{1991}.

\bibitem[{Korte and Vygen(2012)}]{KV2012}
\bibinfo{author}{B.~Korte}, \bibinfo{author}{J.~Vygen},
  \bibinfo{title}{Combinatorial Optimization --- Theory and Algorithms},
  \bibinfo{publisher}{Springer}, \bibinfo{year}{2012}.

\bibitem[{Koshevoy(1999)}]{koshevoy1999}
\bibinfo{author}{G.~A. Koshevoy}, \bibinfo{title}{Choice functions and abstract
  convex geometries}, \bibinfo{journal}{Mathematical Social Sciences}
  \bibinfo{volume}{38}~(\bibinfo{number}{1}) (\bibinfo{year}{1999})
  \bibinfo{pages}{35--44}.

\bibitem[{Manlove(2013)}]{manlove2013}
\bibinfo{author}{D.~F. Manlove}, \bibinfo{title}{Algorithmics of Matching under
  Preferences}, \bibinfo{publisher}{World Scientific}, \bibinfo{year}{2013}.

\bibitem[{McVitie and Wilson(1970)}]{MW70}
\bibinfo{author}{D.~G. McVitie}, \bibinfo{author}{L.~B. Wilson},
  \bibinfo{title}{Stable marriage assignment for unequal sets},
  \bibinfo{journal}{BIT Numerical Mathematics}
  \bibinfo{volume}{10}~(\bibinfo{number}{3}) (\bibinfo{year}{1970})
  \bibinfo{pages}{295--309}.

\bibitem[{McVitie and Wilson(1971)}]{MW1971}
\bibinfo{author}{D.~G. McVitie}, \bibinfo{author}{L.~B. Wilson},
  \bibinfo{title}{The Stable Marriage Problem}, \bibinfo{journal}{Communication
  ACM} \bibinfo{volume}{14}~(\bibinfo{number}{7}) (\bibinfo{year}{1971})
  \bibinfo{pages}{486--490}.

\bibitem[{Murota(2009)}]{murota2009}
\bibinfo{author}{K.~Murota}, \bibinfo{title}{Recent Developments in Discrete
  Convex Analysis}, in: \bibinfo{editor}{W.~J. Cook},
  \bibinfo{editor}{L.~Lov\'{a}sz}, \bibinfo{editor}{J.~Vygen} (Eds.),
  \bibinfo{booktitle}{Research Trends in Combinatorial Optimization},
  \bibinfo{publisher}{Springer}, \bibinfo{pages}{219--260},
  \bibinfo{year}{2009}.

\bibitem[{Murota and Yokoi(2015)}]{MY2015}
\bibinfo{author}{K.~Murota}, \bibinfo{author}{Y.~Yokoi}, \bibinfo{title}{On the
  Lattice Structure of Stable Allocations in a Two-Sided Discrete-Concave
  Market}, \bibinfo{journal}{Mathematics of Operations Research}
  \bibinfo{volume}{40} (\bibinfo{year}{2015}) \bibinfo{pages}{460--473}.

\bibitem[{Oxley(2011)}]{Oxley2011}
\bibinfo{author}{J.~Oxley}, \bibinfo{title}{Matroid Theory},
  \bibinfo{publisher}{Oxford University Press}, \bibinfo{year}{2011}.

\bibitem[{Roth and Sotomayor(1991)}]{RS1991}
\bibinfo{author}{A.~E. Roth}, \bibinfo{author}{M.~Sotomayor},
  \bibinfo{title}{Two-sided Matching: A Study in Game-Theoretic Modeling and
  Analysis}, \bibinfo{publisher}{Cambridge University Press},
  \bibinfo{address}{Cambridge}, \bibinfo{year}{1991}.

\bibitem[{Schrijver(2003)}]{schrijver2003}
\bibinfo{author}{A.~Schrijver}, \bibinfo{title}{Combinatorial Optimization ---
  Polyhedra and Efficiency}, \bibinfo{publisher}{Springer},
  \bibinfo{year}{2003}.

\bibitem[{Yokoi(2014)}]{yokoi2014}
\bibinfo{author}{Y.~Yokoi}, \bibinfo{title}{Matroidal Choice Functions},
  \bibinfo{type}{Mathematical Engineering Technical Reports}
  \bibinfo{number}{METR 2014-32}, \bibinfo{institution}{University of Tokyo},
  \bibinfo{year}{2014}.

\end{thebibliography}

\appendix
\renewcommand*{\thesection}{\Alph{section}}

\section{On the Number of Antimatroids}\label{sec:number}
We show that the number of distinct antimatroids is doubly exponential in the ground set size,
where a ground set is fixed and we count separately isomorphic ones.

\begin{proposition}\label{prop:number}
  The number of antimatroids on a ground set $E$ of size $n$ is at least
  \[
    2^{\binom{n}{\lfloor n/2\rfloor}/2n}=2^{2^{n-O(\log n)}}.
  \]
\end{proposition}

\begin{proof}
The statement follows from the following three facts:
\begin{itemize}
\item
  the number of distinct matroids on the ground set $E$ has the same lower bound (Theorem~\ref{thm:Knuth}), and
\item
  there exists an injection that maps a matroid on $E$ to a path-independent choice function on $E$ (Lemma~\ref{lem:injection}), and
\item
  there exists a one-to-one correspondence between the antimatroids on $E$ and the path-independent choice functions on $E$ (Theorem~\ref{thm:Koshevoy}). \qedhere
\end{itemize}
\end{proof}

A \emph{matroid} is a set system $(E,\mathcal{I})$ satisfying
(I1) $\emptyset\in\mathcal{I}$,
(I2) $X\subseteq Y\in\mathcal{I}$ implies $X\in\mathcal{I}$, and
(I3) $X,Y\in\mathcal{I}$ with $|X|<|Y|$ implies the existence of $e\in Y\setminus X$ such that $X\cup\{e\}\in\mathcal{I}$.
See, e.g., \cite{Oxley2011} for more on the basics.
\begin{theorem}[Knuth~\cite{Knuth1974}]\label{thm:Knuth}
  The number of matroids on a ground set $E$ of size $n$ is at least
  \[
    2^{\binom{n}{\lfloor n/2\rfloor}/2n}=2^{2^{n-O(\log n)}}.
  \]
\end{theorem}

\begin{theorem}[Koshevoy~{\cite[Proposition 2 and Theorems 2 and 3]{koshevoy1999}}]\label{thm:Koshevoy}
  There exists a one-to-one correspondence between the antimatroids on $E$ and the path-independent choice functions on $E$.
\end{theorem}

The following lemma completes the proof of Proposition~\ref{prop:number}.

\begin{lemma}\label{lem:injection}
  Let $E = \{1, \dots, n\}$.
  For a matroid $(E,\mathcal{I})$, define a choice function $C_{\mathcal{I}}\colon 2^E\to 2^E$ by
  \begin{align*}
    C_{\mathcal{I}}(X)\coloneqq \mathrm{lex}\max\{\, X'\subseteq X\mid X'\in\mathcal{I} \,\},
  \end{align*}
  where $\mathrm{lex}\max$ returns the lexicographically maximum member\footnote{Each subset of $E$ is regarded as a sequence of its elements in descending order, e.g., $\{1, 3, 4\}$ is regarded as $(4, 3, 1)$, and the empty set is regarded as the lexicographically minimum.}.
  Then, $C_{\mathcal{I}}$ is path-independent,
  and the map $(E,\mathcal{I}) \mapsto C_{\mathcal{I}}$ is injective.
\end{lemma}

\begin{proof}
As shown in \cite[Example 3.5]{yokoi2014} (see also \cite[Section 6]{fleiner2003fixed}),
this $C_{\mathcal{I}}$ is a matroidal choice function, and hence is path-independent.

We confirm that the map $(E,\mathcal{I}) \mapsto C_{\mathcal{I}}$ is injective.
That is, we show that $C_{\mathcal{I}} \not\equiv C_{\mathcal{I'}}$ if $\mathcal{I} \neq \mathcal{I}'$,
where $C_{\mathcal{I}} \not\equiv C_{\mathcal{I'}}$ means that $C_{\mathcal{I}}(X) \neq C_{\mathcal{I'}}(X)$ for some $X \subseteq E$.
Suppose that $\mathcal{I} \neq \mathcal{I}'$.
Then, without loss of generality, we may assume that there exists a set $I\in \mathcal{I}\setminus\mathcal{I}'$.
This implies that $C_{\mathcal{I}}(I)=I$ while $C_{\mathcal{I'}}(I)\ne I$, and hence $C_{\mathcal{I}} \not\equiv C_{\mathcal{I'}}$.
\end{proof}

\end{document}